\documentclass[11pt]{amsart}
\usepackage[colorlinks,citecolor=red,urlcolor=blue,bookmarks=false,hypertexnames=true]{hyperref} 
\usepackage{amsfonts}
\usepackage{dsfont}
\usepackage{amssymb}
\usepackage{mathrsfs}
\usepackage{nicefrac,eufrak}
\usepackage{bbm}
\usepackage{enumerate}
\usepackage{url,bm}
\usepackage{color}      
\usepackage{capt-of}
\usepackage{graphicx}
\usepackage{cleveref}

\newcommand{\R}{\mathbb{R}}
\newcommand{\N}{\mathbb{N}}
\newcommand{\C}{\mathbb{C}}
\newcommand{\Z}{\mathbb{Z}}
\newcommand{\Q}{\mathbb{Q}}
\newcommand{\F}{\mathcal{F}}

\newcommand{\V}{\mathcal{V}}

\newcommand{\M}{\mathscr{M}}
\newcommand{\B}{\mathcal{B}}
\newcommand{\calA}{\mathcal{A}}
\newcommand{\calH}{\mathcal{H}}

\newcommand{\calC}{\mathcal{C}}

\newcommand{\calU}{\mathcal{U}}

\newcommand{\g}{\gamma}

\newcommand{\La}{\Lambda}

\newcommand{\wh}[1]{\widehat{#1}}

\newcommand{\Aut}{\operatorname{Aut}}

\newcommand{\rar}{\rightarrow}
\newcommand{\bF}{\mathbb{F}}

\newcommand{\p}{\varphi}

\newtheorem{theorem}{Theorem}[subsection]
\newtheorem{lemma}[theorem]{Lemma}
\newtheorem{corollary}[theorem]{Corollary}
\newtheorem{proposition}[theorem]{Proposition}

\newtheorem{remark}[theorem]{Remark}
\newtheorem{example}[theorem]{Example}

\newtheorem{definition}[theorem]{Definition}

\begin{document}

\title{The density of Gabor systems in expansible locally compact abelian groups }
	
	\let\thefootnote\relax\footnote{2020 {\it Mathematics Subject Classification:} Primary 47A15, 43A70, 42C15, 42C40.
		
		{\it Keywords:} Gabor systems, LCA groups, density, expansive automorphism, Bessel sequence, frames.}

	\author{E. King}
	\address{ Mathematics Department, Colorado State University, Fort Collins, CO, USA}
	\email{emily.king@colostate.edu}
	
	\author{R. Nores}
	\address{ Departamento de Matem\'atica, Universidad de Buenos Aires,
		Instituto de Matem\'atica "Luis Santal\'o" (IMAS-CONICET-UBA), Buenos Aires, Argentina}
	\email{rnores@dm.uba.ar}

	\author{V. Paternostro}
	\address{ Departamento de Matem\'atica, Universidad de Buenos Aires,
		Instituto de Matem\'atica "Luis Santal\'o" (IMAS-CONICET-UBA), Buenos Aires, Argentina}
	\email{vpater@dm.uba.ar}

	\begin{abstract}
		
		We investigate the reproducing properties of Gabor systems within the context of expansible groups. These properties are established in terms of density conditions. The concept of density that we employ mirrors the well-known Beurling density defined in Euclidean space, which is made possible due to the expansive structure. Along the way, for groups with a compact open subgroup, we demonstrate that modulation spaces are continuously embedded in Wiener spaces. Utilizing this result, we derive the Bessel condition of Gabor systems. We also provide a straightforward proof of the density result for Gabor frames, utilizing a comparison theorem for coherent frames.
		
	\end{abstract}

	\maketitle

\section{Introduction}

A Gabor system is a set of functions in $L^2(\mathbb R^d)$ that are obtained by translating a single window in time and frequency along a set $\Lambda\subseteq \mathbb R^{2d}$. Their structure make them of particular importance in many applications such as wireless communication,  analysis and description of speech or music signals, and more (see, e.g., \cite{GRO, GH99, SH03}).
Because of this,  it is important to study their reproducing properties, that is, to understand which properties of the generating window and the set $\Lambda$ of time-frequency shifts guarantee the Gabor system to be an orthonormal basis, a Riesz basis, a frame, or a Bessel sequence. When the set $\Lambda$ does not have an algebraic structure, these reproducing properties are stated in terms of its Beurling density.

Consider a  set $\Lambda \subset \R^d$. The upper and lower Beurling density of $\Lambda$, denoted by $D^+(\Lambda)$ and $D^-(\Lambda)$, respectively, captures its asymptotic behavior within balls of varying radii. In mathematical terms, they are given by:
$$D^+(\Lambda)=\limsup_{r\to\infty} \sup_{x\in\R^d}\frac{\#(\Lambda\cap B_r(x))}{r^d},\qquad\textrm {
and }
\qquad D^-(\Lambda)=\liminf_{r\to\infty} \inf_{x\in\R^d}\frac{\#(\Lambda\cap B_r(x))}{r^d},$$
where $B_r(x)$ denotes the ball centered in $x$ with radius $r$ and $\#$ indicates the cardinality of a set.
A thorough survey on the results about density of Gabor systems is \cite{H07}.

Beurling density is a concept that is also related to conditions for sampling and interpolation. See, for example, Landau's result in \cite{Lan67}. In order to explore the validity of Landau's result in the context of locally compact abelian groups that are compactly generated, in \cite{GK08} Gr\"ochenig, Kutyniok, and Seip provide a definition of Beurling density by means of a comparison with a canonical lattice of reference. 
We refer to  \cite{AAC15} for related results on the existence of sampling and interpolation sets near to the critical density in LCA groups.

In this paper we focus on studying Gabor systems in terms of density within the framework of expansible locally compact abelian groups, that is, locally compact abelian (LCA) groups which have a compact open subgroup and an expansive automorphism. (See Definition \ref{def:expansive} for more details). One example of a such an expansible LCA group is the space of $p$-adic numbers $\mathbb Q_p$.

This type of LCA group does not possess any lattices, and thus, the notion of density given in \cite{GK08} needs to be reformulated. However, the presence of the expansive automorphism allows us to generalize the concept of density in the spirit of Beurling density in $\mathbb R^d$. A similar approach was given in \cite{SK20}.

Given an expansible LCA group $G$, a window $\varphi\in L^2(G)$ and a set $\Lambda\subseteq G\times \wh G$, where $\wh G$ denotes the Pontryagin dual group of $G$,  we explore whether or not the Gabor system generated by $\varphi$ with time-frequency shifts along $\Lambda$, $S(\varphi, \Lambda)$,  is a Bessel sequence or a frame of $L^2(G)$ depending on the density of $\Lambda$.

We first  move to a more abstract setting where we deal with unitary and projective representations of a locally compact group having a compact open subgroup. We  based our analysis in the notion of modulation spaces introduced in \cite{fe83-4} (see also \cite{FG1, FG2,fe03}). We prove that the generalized wavelet transform (voice transform) induced by the (unitary or projective) representation continuously maps the  modulation space of order $p$ into the Wiener space $W(\calC, \ell^p)$. We then apply these results to the case were the projective representation is the one given by time-frequency translations, resulting in Gabor systems.

Returning to Gabor systems in expansive groups, we demonstrate that if $S(\varphi, \Lambda)$ is a Bessel sequence in $L^2(G)$, then $\Lambda$ must have finite density. Conversely, when $\Lambda$ has finite density and $\varphi$ exhibits some decay, $S(\varphi, \Lambda)$ is a Bessel sequence. This latter result leverages the continuous embedding of modulation spaces into Wiener spaces.

Our findings serve as analogues for expansive LCA groups to \cite[Theorem 3.1]{CDH99} and \cite[Theorem 12]{H07}, which are set in the Euclidean context.

Finally, we derive a density result for Gabor systems, providing a straightforward proof based on the comparison theorem \cite[Theorem 4]{GK08}.

The article is structured as follows: In Section \ref{sec:preliminaries}, we recall certain definitions about LCA groups  and  expansive automorphisms as well as what we will need about frame theory.
We investigate in Section \ref{sec:density} the concept of density. We establish that it is (essentially) independent of the chosen automorphism used to define the density and show an equivalent condition for a set to have finite upper density. 
Moving to Section \ref{sec:modulation-spaces}, we use  the theory of modulation spaces introduced by  Feichtinger \cite{fe83-4} and developed further by Feichtinger and Gr\"ochenig in \cite{FG1, FG2}. When the underlying locally compact group has a compact open   subgroup, we establish one of the pivotal results of this work, namely, the continuous inclusion of modulation spaces into Wiener spaces through the generalized wavelet transform. Finally, in Section \ref {sec:Gabor} we prove that the classical properties of Gabor systems hold true in our context.

\section{Preliminaries}\label{sec:preliminaries}
In this section we fix the context where we will work in. We will also recall some aspects of frame theory that we will need.

\subsection{LCA groups}
A \emph{locally compact abelian group} (\emph{LCA group}) $G$  is an abelian group which is also a locally compact topological space such that both multiplication and inversion are homeomorphisms of the space. We will always assume that the topology is Hausdorff.   See, e.g., \cite{HR69} for a general reference

Given an LCA group $G$ written additively, we denote by $\wh G$ its Pontryagin  dual group. By $m_G$ we denote 
        a Haar measure associated to $G$ (with the desired normalization defined below). Since the dual of the dual group is topologically isomorphic to the original group, for  $\xi\in\wh G$ and $x\in G$ we write $\langle x,\xi\rangle $ to indicate the character $\xi$ applied to $x$ (i.e. $\xi(x)$) or the character  $x$ applied to $\xi$. For a subgroup  $H\subseteq G$, its annihilator is denoted by $H^\perp$ and is defined as 
$H^\perp=\{ \xi\in\wh G:\, \langle h,\xi\rangle=1, \,\forall \,h\in H\}$. It is well known that if $H$ is closed, $H^\perp$ is a closed subgroup of $\wh G$. 
  
For a closed subgroup $H$ of $G$, the dual of the quotient group $G/H$,  $\widehat{G/H}$ is algebraically and  topologically isomorphic to $H^\perp$ and $\wh H$ is algebraically and  topologically isomorphic to $\wh G/H^\perp$.

When $H\subseteq G$ is a compact open subgroup, then so is 
$H^\perp\subseteq\wh G$. As a consequence,   the quotients $G/H$ and 
$\wh G/H^\perp$ are discrete abelian groups. 

For an LCA group $G$ with a compact open subgroup $H$, we consider the 
following normalization of the Haar measures involved: we fix $m_G$ such that $m_G(H)=1$ and $m_{\wh G}(H^\perp)=1$. As explained in 
\cite[Comment (31.1)]{HR69}, this choice guarantees that the Fourier transform between $L^2(G)$ and $L^2(\wh G)$ is an isometry. We take $m_H=m_G\mid_H$, 
 $m_{\wh H}=m_{\wh G}\mid_{H^\perp}$ and $m_{G/H}$ and $m_{\wh G/H^\perp}$ to be the counting measures. Then, the Fourier transforms  between  $L^2(H)$ and 
$L^2(\wh G/H^\perp)$ and between $L^2(G/H)$ and $L^2(H^\perp)$ are isometries.

\subsection{Expansive automorphisms}
Let $G$ be an LCA group. The group of  homeomorphic automorphisms of $G$ into itself 
is  denoted by $\Aut(G)$. For a given $A\in \Aut(G)$, the measure $\mu_A$ defined by 
$ \mu_A(U)=m_G(AU)$, where $U$ is a Borel set of $G$, is a non-zero Haar measure on $G$. Therefore, there is a unique positive number $|A|$, the so-called {\it modulus of $A$}, such that $\mu_A=|A|m_G$.

For $A\in \Aut(G)$, there is an adjoint $A^* \in \Aut(\wh G)$ defined as $\langle Ax,\g\rangle=\langle x, A^*\g\rangle$ for all $x\in G$ and $\g\in\wh G$. It holds that $|A^*|=|A|$.

We next present the definition of expansive automorphisms as given in \cite{BEN04}, where they were used to define a wavelet theory over local fields. See also \cite{SK20}.

\begin{definition}\cite[Definition 2.5]{BEN04}\label{def:expansive}
Let $G$ be an LCA group and $H\subseteq G$ a compact open subgroup, and let $A\in \Aut(G)$.  We say that $A$ is \emph{expansive with respect to $H$} if the following two conditions hold:
\begin{enumerate}
\item $H\subsetneq AH$;
\item $\bigcap_{n\leq 0}A^nH=\{0\}.$
\end{enumerate}
When $H$ is fixed or clear from the context, we will simply say that $A$ is expansive. 
\end{definition}

There exist many groups $G$ with expansive automorphism $A$. We now give several examples.

\begin{example}\cite[Example 2.10]{BEN04}  Let $p$ be a prime number.  Define the $p$-adic valuation over $\Q$ as $|p^rx|_p=p^{-r}$ for all $r\in\Z$ and all $x\in\Q$ such that the numerator and denominator of $x$ are both relatively prime to $p$.  Then $\Q_p$ is the completion of $\Q$ with respect to the $p$-adic valuation.  Each element of $\Q_p$ may be represented as a Laurent series 
$$\Q_p=\left\{ \sum_{n\geq n_0} a_np^n: \, n_0\in\Z\,\,\mathrm{and}\,\, a_n\in\{0,1,\hdots,p-1\}\right\},$$
where addition ``carries'' rather than is modular.  That is,
\[
(p-1)p + 1p = p^2 \neq 0p.
\] 
Then $\Q_p$ is a locally compact abelian group with group operation addition and topology defined via the $p$-adic valuation.  The Laurent series are not formal as they converge in the topology.

Then the $p$-adic integers $\Z_p$, defined as the set of power series
$$\Z_p=\left\{ \sum_{n\geq 0} a_np^n: \, a_n\in\{0,1,\hdots,p-1\}\right\},$$
is a compact open subset of $\Q_p$.  Further characterizations of $\Z_p$ are that $\Z_p$ is the unit ball (with respect to $p$-adic valuation) of $\Q_p$ and the closure of $\Z$ in $\Q_p$.

The $p$-adic numbers are self-dual.  Let $\{\cdot\}: \Q_p \rar \Q_p$ be defined as
\[
\left\{  \sum_{n\geq n_0} a_np^n \right\} =  \sum_{n=n_0}^{\max\{0,n_0\}} a_np^n.
\]
Then each $y \in \Q_p$ defines an element of $\wh{\Q}_p$ as $\langle \cdot, y\rangle = \exp(2\pi i \left\{ \cdot y\right\})$.

If we consider $A:\Q_p\to\Q_p$ to be the morphism given by $Ax=p^{-1}x$, then $A$ is an automorphism of $\Q_p$, and it is easy to see that is expansive with respect to $\Z_p$. 
\end{example}

\begin{example}\cite[Example 2.11]{BEN04} Let $p$ be a prime number, where $\bF_p$ is the field of order $p$.  The additive group of the field $\bF_p((t))$ of formal Laurent series in variable $t$:
$$\bF_p((t))=\left\{ \sum_{n\geq n_0} a_nt^n: \, n_0\in\Z\,\,\mathrm{and}\,\, a_n\in\bF_p \right\},$$
where addition is modular rather than ``carries'':
\[
(p-1)t + 1t = 0t =0
\] 
is an LCA group with respect to the topology defined from an analog of the $p$-adic valuation.  The set of formal power series
$$\bF_p[[t]]=\left\{ \sum_{n\geq 0} a_nt^n: \, a_n\in\bF_p \right\}$$
is a compact open subgroup.  One possible expansive automorphism is multiplying by $t^{-1}$.  However, the structure of $\bF_p((t))$ yields a richer collection of automorphisms than for $\mathbb{Q}_p$.
\end{example}

Both classes of the LCA groups above are also fields.  Further examples may be formed by considering the additive groups of finite field extensions or vector spaces over the above examples.

\begin{example}\cite[Example 2.14]{BEN04} \label{Ex:counterexample} Let $G_1$ be an LCA group with a compact open subgroup $H_1$ and $G_2$ be a nontrivial discrete abelian group. If we consider $G=G_1\times G_2$, then $G$ is an LCA group with a compact open subgroup $H=H_1\times\{0\}$. Further, the annihilator $H^{\perp}$ is $H_1\times\wh{G_2}$ and if $A_1$ is an automorphism of $G_1$ which is expansive  with respect to $H_1$, then $A=A_1\times id_{G_2}$ is an automorphism of $G$ which is expansive with respect to $H$. However, in this case, the union of all positive iterates $A^nH$ can not cover $G$ since $AH=A_1H_1\times\{0\}$. 
\end{example}

Following \cite{SK20}, we shall call an  LCA group which admits an expansive automorphism an {\it expansible} group.  
Expansiveness can also be characterized by the action of the adjoint automorphism as the next lemma shows, whose proof can be found in \cite[Lemma 2.6]{BEN04}.

\begin{lemma} Let $G$ be an LCA with a compact open subgroup $H\subseteq G$, and let $A\in \Aut(G)$.
\begin{itemize}
\item [$(i)$]  $H\subseteq AH \hspace{.2cm}\mbox{if and only if}\hspace{.2cm} H^{\perp}\subseteq A^*H^{\perp}$.
\item[$(ii)$]  $H\subsetneq AH\hspace{.2cm} \mbox{if and only if}\hspace{.2cm} H^{\perp}\subsetneq A^*H^{\perp}$.
\item[$(iii)$]  If $H\subseteq AH$, then
\begin{align} \label{exp}\bigcap_{n\leq0}A^nH=\{0\} \iff \bigcup_{n\geq0}A^{*n}H^{\perp}=\wh{G}.\end{align}
\end{itemize}
\end{lemma}

Given an LCA group $G$  with a compact open subgroup $H$ and expansive automorphism $A: G \to G$, we define for $n \in \Z$ and $x \in G$,
\[
Q_n(x) = x + A^n H.
\]
One may think of this as the ``ball'' with ``center'' $x$ and ``radius'' $|A|^n$, keeping in mind that if $y \in Q_n(x)$ then $Q_n(x) = Q_n(y)$, so the choice of ``center'' is not unique.  In the case that $G = \Q_p$, $H= \Z_p$ , and $A$ is multiplication by $1/p$, the $Q_n(x)$ are precisely the balls in the metric induced by the $p$-adic valuation.  Note that each $Q_n(x)$ is a compact open subset of $G$ and also a coset of $A^n H$ in $G$.
Moreover, by \cite[Theorem 4.5]{HR62}, $\{Q_n(x)\}_{n\in\Z, x\in G}$ is a basis for the topology of $G$.

In this paper, we will  be dealing with the specific scenario of forming  ``balls'' in $G \times \wh{G}$, where $G$ satisfies the hypothesis above.  In that case, we write for $n \in \Z$ and $(x,\gamma) \in G \times \widehat{G}$
\begin{align}\label{def:bola}
Q_n(x,\gamma) & = (x,\gamma) + \left(A^n \otimes (A^\ast)^n\right)  (H \times H^\perp) \nonumber\\ 
&= \left(x + A^n H\right) \times \left(\gamma + (A^\ast)^n H^\perp\right)  = Q_n(x) \times Q_n(\gamma).
\end{align}
Note that in this case, the ``radius" of $Q_n(x,\gamma)$ is $(|A||A^*|)^n=|A|^{2n}$.

When  $A$ is expansive with respect to $H$ and  $A^*$ is expansive with respect to $H^\perp$, we have that $\{Q_n(x,\gamma)\}_{n\in\Z, x\in G,\gamma\in\widehat{G}}$ is a basis for the topology of $G\times\widehat{G}$. This is because for $n\leq m, x\in G$ and $\gamma\in\wh G$, $Q_n(x) \times Q_m(\gamma)\subseteq Q_m(x,\gamma)$.

\subsection{Frames and Riesz bases}

Let  $\{\varphi_i\}_{i\in I}$ be a family of elements  in a separable  Hilbert space $\calH$.  It is  said that $\{\varphi_i\}_{i\in I}$ is  a \emph{frame for $\calH$} if there exist constants $A$,$B>0$ such that
\begin{equation}\label{framecondition} A\|f\|^2\leq \sum_{i\in I} |\langle f,\varphi_i\rangle|^2\leq B \|f\|^2 \,\,\,\,\, \forall f\in\calH.\end{equation}
The constants $A, B$ are called \emph{frame bounds}. The \emph{frame operator} defined as $Sf=\sum_{i\in I}\langle f, \varphi_i\rangle \varphi_i$ for $f\in \calH$  is a bounded, invertible, and positive operator  from  $\calH$ onto itself. This provides the well known frame decomposition 
$$f=S^{-1}Sf=\sum_{i\in I} \langle f,\varphi_i\rangle \phi_i \,\,\,\,\, \forall f\in\calH,$$
where $\phi_i=S^{-1}\varphi_i$. The family $\{\phi_i\}_{i\in I}$ is also a frame for $\calH$, which is called the \emph{canonical dual frame}, and has frame bounds $B^{-1}$, $A^{-1}$. Any other frame $\{\widetilde{\phi_i}\}_{i\in I}$ satisfying 
$f=\sum_{i\in I} \langle f,\varphi_i\rangle \widetilde{\phi_i} \,\,\,\,\, \forall f\in\calH,$ is called a \emph{dual frame}, and it is well known that a frame can have dual frames besides the canonical one (typically infinitely many).  For a general reference on frame theory, see, e.g., \cite{Chr16} and the references therein.

Riesz bases are special cases of frames and can be characterized as those frames which are biorthogonal to their canonical dual frame, i.e., such that $\langle \varphi_i,\phi_j\rangle=\delta_{ij}$.

A family $\{\varphi_i\}_{i\in I}$ which satisfies the right  inequality in \eqref{framecondition} (but possibly not the left) is called a \emph{Bessel sequence} for $\calH$.

\section{Density with respect to expansive automorphisms}\label{sec:density}

In this section, using the balls defined above,  we will consider the concept of density, which will extend that of the well known  Beurling density for Euclidean spaces $\R^d$. 
This concept has been considered before in LCA groups that are compactly generated in \cite{AAC15, GK08}, however, with another  approach. In \cite{SK20} the authors work with  the same  density as here, but they consider a slightly different notion of  expansive automorphism. In fact, their automorphism only satisfies condition $(2)$ of Definition \ref{def:expansive}.

\begin{definition}\label{def:separated}
Let $G$ be an LCA group, $H\subseteq G$ a compact open subgroup and $A\in 
\Aut(G)$ expansive with respect to $H$. For $n\in\Z$, a set (countable or uncountable)  $\Lambda\subseteq G$ is 
said to be {\it$(A,n)$-uniformly separated} if 
$\#\{\Lambda\cap Q_n(x)\}\leq 1$ 
for all $x\in G$. 
We say that $\Lambda$ is simply {\it uniformly separated} if it is $(A,n)$-uniformly separated for some $n\in\Z$.
Additionally, $\Lambda$ is said to be {\it $A$-separated} if it is a finite union of uniformly separated sequences. 
\end{definition}

 Recall that if $G$ is an LCA group and  $H\subseteq  G$ is a subgroup,  a {\it section} of a quotient group $G/H$ is a measurable set of representatives, and it contains exactly one element of each coset. 
As stated above, if $H\subseteq  G$ is a compact open subgroup, since $H^{\perp}$ is also compact, then $G/H$ is a discrete group. Thus, every section $C\subseteq G$ for the quotient $G/H$ must be discrete as well. This is because for $x\in C$, $x=C\cap (x+H)$ and then, since $x+H$ is an open set in $G$, $C$ is discrete with respect to the topology of $G$. 

With this in mind, we can say that $\Lambda$ is $(A,n)$-uniformly separated if and only if 
$\#\{\Lambda\cap Q_n(A^nc)\}\leq 1$ 
for all $c\in C$. This is a direct consequence of the fact that for every $n\in\Z$,  $\{Q_n(A^nc)\}_{c\in C}$ is a partition of $G$, that is, $G=\bigcup_{c\in C}Q_n(A^nc)$  where the union is disjoint,  and that $Q_n(A^nc)=Q_n(x)$ for every $x\in Q_n(A^nc)$.

\begin{definition}\label{def:density}
Let $G$ be an LCA group, $H\subseteq  G$ a compact open subgroup and $A\in 
\Aut(G)$ expansive with respect to $H$. For a sequence  $\Lambda\subseteq G$, the \emph{upper and lower Beurling density of $\Lambda$} are  defined by 
$$D^+_A(\Lambda):=\limsup_{n\to +\infty}\frac1{|A|^n}\max_{x\in G}\#\{\Lambda\cap Q_n(x)\},$$
and 
$$D^-_A(\Lambda):=\liminf_{n\to +\infty}\frac1{|A|^n}\min_{x\in G}\#\{\Lambda\cap Q_n(x)\},$$ respectively. 
If $D^+_A(\Lambda)=D^-_A(\Lambda)$ we say that $\Lambda$ has \emph{uniform density} $D_A(\Lambda)=D^+_A(\Lambda)=D^-_A(\Lambda)$.
\end{definition}

The analogy with the Beurling density defined in $\R^d$ \cite{CDH99, H07,SK20} is clear from the definition noting that, since $m_G$ is invariant under translations and $m_G(H)=1$, we have that $|A|^n=m_G(A^nH)=m_G(Q_n(x))$ for any $x\in G$.

At first glance, the Beurling density seems to depend on the automorphism. However, if both the automorphism and their adjoint are expansive, the density becomes independent of the automorphism choice, as we show next. First, we observe some properties about sections that will be useful.  

\begin{remark}\label{rem:sections}
Given $A\in \Aut(G)$ expansive  with respect to a compact open subgroup $H$ of $G$, the quotient $AH/H$ must be finite. This is because $AH$ is compact and  then, it may be covered by finite (disjoints) cosets $x+H$. Let $C_0\subseteq AH$ be a finite section for $AH/H$. Therefore, $AH=\bigcup_{c_0\in C_0}H+c_0$, and then we have that $m_G(AH)=\#C_0$ which implies $|A|=\#C_0$. 
Moreover, since $AH\subseteq  G$ is also a compact open subgroup, we can take $C_1$ a discrete section for $G/AH$.
An easy computation shows that the set $C:=C_0+C_1$ must be a (discrete) section for $G/H$. From now on, we will consider sections for $G/H$ of this form.
\end{remark}

\begin{lemma} Let $G$ be an LCA group, $H\subseteq  G$ a compact open subgroup and $A$, $B\in \Aut(G)$ be expansive with respect to $H$ such that $A^*$ and $B^*$ are expansive with respect to $H^\perp$ as well. Then, for every sequence $\Lambda\subseteq G$,  
$$D^+_A(\Lambda)=D^+_B(\Lambda)\quad\textrm { and }\qquad D^-_A(\Lambda)=D^-_B(\Lambda).$$
\end{lemma}

\begin{proof} Because of the expansiveness of $A^*$ we know by (\ref{exp}) that $G=\bigcup_{n\geq0}A^nH$. Let $n\in\N$ be fixed. Then there exists $k_0=k_0(n)\in\N$ such that $B^nH\subseteq A^{k_0}H$ because $B^nH$ is compact and $\{A^kH\}_{k\in\N}$ is an open cover of nested sets.

Let $X=\{x_i\}_{i\in I}\subseteq A^{k_0}H$ be a section of $A^{k_0}H/B^nH$. Then 

$$A^{k_0}H=\bigcup_{i\in I} (B^nH+x_i), $$

and for $x\in G$, 
$$Q_{k_0}^A(x)=A^{k_0}H+x=\bigcup_{i\in I} (B^nH+x_i)+ x=\bigcup_{i\in I} Q_n^B(x_i + x).$$

By taking intersections with $\Lambda$ and taking the maximum cardinality over $G$, we obtain the following inequality:
\begin{align*}
\max_{x\in G}\#\left\{Q_{k_0}^A(x)\cap\Lambda\right\}&=\max_{x\in G} \#\left\{ \bigcup_{i\in I} Q_n^B(x_i + x)\cap\Lambda\right\}\\
&\leq \sum_{i\in I} \max_{x\in G} \#\left\{Q_n^B(x_i + x)\cap\Lambda\right\} \\
&=\sum_{i\in I} \max_{x\in G} \#\left\{Q_n^B(x)\cap\Lambda\right\} = (\#I) \max_{x\in G} \#\left\{Q_n^B(x)\cap\Lambda\right\}. 
\end{align*}

Furthermore, $\#(A^{k_0}H/H)=\#(A^{k_0}H/B^nH) \, \#(B^nH/H)$ and then
$$\#I=\#(A^{k_0}H/B^nH)=\frac{|A|^{k_0}}{|B|^n}.$$

Finally, we obtain 

$$\frac{\max_{x\in G}\#\{Q_{k_0}^A(x)\cap\Lambda\}}{|A|^{k_0}}\leq \frac{|A|^{k_0}}{|B|^n} \frac{\max_{x\in G} \#\{Q_n^B(x)\cap\Lambda\}}{|A|^{k_0}}=\frac{\max_{x\in G} \#\{Q_n^B(x)\cap\Lambda\}}{|B|^n}.$$

If we take $\limsup_{n\to\infty}$  in the last inequality we have that $D^+_A(\Lambda)\leq D^+_B(\Lambda)$. By using the expansiveness of $B^*$ and using the same reasoning, we obtain that $D^+_A(\Lambda)=D^+_B(\Lambda)$. A similar argument proves that $D^-_A(\Lambda)=D^-_B(\Lambda)$.

\end{proof}

As a consequence of the above result and in order to keep the exposition as clear as possible, we choose to omit the subscript $A$ in the density and simply write $D^+(\Lambda)$, $D^-(\Lambda)$ and $D(\Lambda)$.

The next lemma shows a characterization of the sequences $\Lambda$ whose upper density is finite and provides a valid version of  \cite[Lemma 2.3]{CDH99} in this context. Our proof is based on the group structure. 
In \cite[Theorem 3.7]{SK20},  the authors proved the same result based on \cite[Lemma 2.3]{SK20}, which  is not satisfied in our case. (See Example \ref{Ex:counterexample}  and \cite[Lemma 2.3, item (ii)]{SK20}).

\begin{lemma}\label{lem:density}
Let $G$ be an LCA group, $H\subseteq G$ a compact open subgroup and $A\in 
\Aut(G)$ expansive with respect to $H$. If $\Lambda\subset G$, then the following conditions are equivalent:
\begin{enumerate}
\item[$(i)$] $D^+(\Lambda)<\infty$;
\item[$(ii)$] For some $n\in\Z$ there exists $N_n>0$ such that 
$\#\{\Lambda\cap Q_n(A^nc)\}\leq N_n$ for all $c\in C$, where $C$ is a section for $G/H$. 
\end{enumerate}
Additionaly, if $(ii)$ holds for some $n\in\Z$, it holds for every $n\in\Z$.
\end{lemma}

\begin{proof}
$(i)\Rightarrow (ii)$. It is obvious from the definition of $D^+(\Lambda)$.\\
$(ii)\Rightarrow (i)$. Let $n\in\Z$ and $N_n>0$ be such that 
$\#\{\Lambda\cap Q_n(A^nc)\}\leq N_n$ for all $c\in C$ and label for every $c\in C$,
$\Lambda\cap Q_n(A^nc)=\{\lambda_{1,c}, \cdots, \lambda_{r,c}\}$ with $r=r(c)\leq 
N_n$. 
Now, for $j\in\{1,\cdots, N_n\}$ and $c_0\in C_0$ set 
$\Lambda_{j, c_0}:=\{ \lambda_{j,c_0+c_1}:\,c_1\in C_1\}$. By construction, each 
element of $\Lambda_{j, c_0}$ lies in a different coset of $G/A^nH$. Then, since there 
are at most $|A|N_n$ many $\Lambda_{j, c_0}$ sets, we have that any coset of  $G/A^nH$ contains at most $|A|N_n$ elements of $\Lambda$.

Now, since 
$$ A^{n+1}H=A^n(AH)=\bigcup_{c_0\in C_0}\left( A^nH+A^nc_0 \right)$$
and 
$$G=A^nG=\bigcup_{c_0\in C_0, c_1\in C_1}\left(A^nH+A^nc_0+A^nc_1\right)=
\bigcup_{c_1\in C_1}\left(\bigcup_{ c_0\in C_0}\left(A^nH+A^nc_0\right)+A^nc_1\right),$$
each coset of $G/(A^{n+1}H)$ has at most $|A|^2N_n$ elements of $\Lambda$. 
Continuing by induction, we see that if $m>n$, then each coset of $G/(A^{m}H)$ has at most $|A|^{m-n+1}N_n$ elements of $\Lambda$. 

As a consequence, 
\begin{align*}
D^+(\Lambda) &= \limsup_{m\to +\infty}\frac1{|A|^m}\max_{x\in G}\#\{\Lambda\cap Q_m(x)\}\\
&\leq \limsup_{m\to +\infty}\frac1{|A|^m} |A|^{m-n+1}N_n\\
&= \frac{N_n}{|A|^{n-1}}<\infty.
\end{align*}

Suppose now that $(ii)$ fails for some $n\in\Z$. Then, $\max_{x\in G}\#\{\Lambda\cap Q_m(x)\}=+\infty$ for every $m\geq n$. As a consequence, $D^+(\Lambda)=+\infty$. 
This completes the proof. 
\end{proof}

\begin{remark}\label{rem:separation}\noindent
Note that if $\Lambda\subset G $ and $D^+(\Lambda)<\infty$, we can deduce from the proof of Lemma \ref{lem:density} that for every fixed $n\in\Z$, 
$\Lambda =\bigcup_{j\in\{1,\cdots, N_n\}, c_0\in C_0}\Lambda_{j,c_o}$ where the union is disjoint. Moreover, every  set $\Lambda_{j,c_o}$ is a uniformly separated sequence because we saw that each element of  $\Lambda_{j,c_o}$ lies in a different coset of $G/A^nH$ and then $\#\{\Lambda_{j,c_o}\cap\left[A^n H+A^n c\right]\}\leq 1$ for every $c\in C$.
Therefore, when $D^+(\Lambda)<\infty$ we have that $\Lambda$ must be a finite union 
of uniformly separated sequences;  that is, $\Lambda$ must be $A$-separated.
Since uniformly separated sequences must be countable, $\Lambda$ must be countable as well. 

\end{remark}
 
\begin{example}\label{ex:densidadC}
Let $G$ be an LCA group, $H\subseteq G$ a compact open subgroup  and $C$  section of the quotient $G/H$. Consider $\Lambda:=C$. Then,  we have that  for all $c\in C$, $\#\{C\cap Q_0(c)\}=1$ and proceeding as in the proof of Lemma \ref{lem:density} we get  
$\#\{C\cap Q_n(A^nc)\}=|A|^n$ for all $n\in\Z$. Therefore, $D^+(C)=D^-(C)=1=D(C)$. 
\end{example}

\section{Modulation spaces on locally compact  groups}\label{sec:modulation-spaces}

In this section, we consider $G$  a locally compact group (non necessarily abelian), with a compact open subgroup $H$. Using definitions and lemmas from the seminal papers \cite{FG1, FG2, fe81-2, fe83-4,fe03}, we will prove that the generalized wavelet transform continuously maps modulation spaces into Wiener spaces. 

Let $G$ be a locally compact  group with right Haar measure $m_G$. To emphasize that $G$ needs not be abelian, in this  section we will use multiplicative notation for the operation. Given  $\calH$  a Hilbert space, a unitary representation of $G$ on $\calH$ is a continuous homomorphism    $\pi:G\to\calU(\calH)$. For $f, g\in\calH$ we consider $\V_gf:G\to\C$, the \emph{generalized wavelet transform} (also called \emph{voice transform} or \emph{representation coefficients}) of $f$ with respect to the window $g$ defined by
$$\V_gf(x):=\langle f,\pi(x)g\rangle,\quad x\in G.$$

The set of \emph{analyzing vectors} on $G$ is given by
\begin{equation}\label{eq:DefA}
\calA:=\{ g\in \calH: \V_gg\in L^1(G)\}.
\end{equation}
Note that for $g\in\calH$ and $x\in G$, since $\pi$ is a representation, $ \V_{\pi(x)g}(\pi(x)g)(y)=\V_gg(x^{-1}yx)$ for all $y\in G$. Thus,  $\calA$ is invariant under $\pi$.  Then, when $\pi$ is an irreducible unitary representation, that is,  without  proper invariant subspaces, $\calA$ must be a dense linear subspace of $\calH$.

In the remainder of this section, we assume that $H\subseteq G$ is a compact open subgroup. Examples of non-abelian groups with that property follows.
\begin{example}
Fix $p$ prime and $n \geq 2$.  Then the general linear group $\operatorname{GL}_n(\Q_p)$ is a non-abelian locally compact group with well-understood representation theory, and  $\operatorname{GL}_n(\Z_p)$ is a compact open subgroup \cite{BuHe06}. 
\end{example}
We denote the space of continuous functions on $G$ as $\calC:=\calC(G)$. For $\varphi\in \calC$ and every $x\in G$ we see that $\|\chi_{Hx}\cdot \varphi\|_{\infty}=\sup_{y\in Hx}|\varphi(y)|<\infty$, where $\chi_A$ is the characteristic function of $A$ that takes the value $1$ for $x \in A$ and $0$ otherwise.

For $1\leq p<\infty$, the Wiener space  $W(\calC, L^p)$ is defined as 
$$W(\calC, L^p):=\{ \varphi\in\calC : \int_{G} \|\chi_{Hx}\cdot \varphi\|_{\infty}^p dm_G(x) <\infty\}.$$
It turns out that $W(\calC, L^p)$ equipped with the norm $\|\varphi\|_{W(\calC, L^p)}:=\left( \int_{G} \|\chi_{Hx}\cdot \varphi\|_{\infty}^p dm_G(x)\right) ^{1/p}$ is a Banach space.
 
Additionally, if $C$ is  a section of $G/H$ and $1\leq p<\infty$, let us denote by $W(\calC, \ell^p)$ the space 
$$W(\calC,\ell^p):=\{\varphi\in\calC : \sum_{x\in C}\|\chi_{Hx}\cdot \varphi\|_{\infty}^p<\infty\}.$$
We note that $W(\calC, \ell^p)$ does not depend on the choice of the section $C$.  Furthermore, it holds that $W(\calC,\ell^p)$ endowed with  the norm given by $\|\varphi\|_{W(\calC, \ell^p)}:=\left(  \sum_{x\in C}\|\chi_{Hx}\cdot \varphi\|_{\infty}^p\right) ^{1/p}$ is a Banach space.

Moreover, it can be seen that actually  $W(\calC, \ell^p) =  W(\calC, L^p)$ and the norms $\|\cdot\|_{W(\calC, L^p)}$ and $\|\cdot\|_{W(\calC, \ell^p) }$ coincide as we show in the next lemma. 

\begin{lemma} Let $G$ be a locally compact group  with a compact open subgroup. Then,  $W(\calC, \ell^p)=  W(\calC, L^p)$. Moreover, for $\varphi\in \calC$,
$$ \|\varphi\|_{W(\calC, \ell^p)}=\|\varphi\|_{W(\calC, L^p)}.$$

\end{lemma}

\begin{proof} Let $\varphi\in \calC$. Since  $G=\bigcup_{c\in C} Hc$ where the union is disjoint, $C$ is a section of $G/H$ and $m_G(H)=1$, we can write
\begin{align*}
\|\varphi\|_{W(\calC, L^p)}^p&=\int_G \|\chi_{Hx}\cdot\varphi\|_{\infty}^p dm_G(x)=\sum_{c\in C} \int_{Hc}\|\chi_{Hx}\cdot\varphi\|_{\infty}^p dm_G(x) \\
&=\sum_{c\in C} \int_{Hc}\|\chi_{Hc}\cdot\varphi\|_{\infty}^p dm_G(x)
=\sum_{c\in C} \|\chi_{Hc}\cdot\varphi\|_{\infty}^p
\\&=\|\varphi\|_{W(\calC, \ell^p)}^p.
\end{align*}
Then $\varphi\in W(\calC, \ell^p)$ if and only if $\varphi\in W(\calC, L^p)$.
\end{proof}

We now  consider the following  subset of $\calA$,
\begin{equation}\label{eq:DefB}
\B=\{g\in \calH: \V_gg\in W(\calC,L^1)\},
\end{equation}
which turns out to be    also invariant under $\pi$. As before, when $\pi$ is irreducible, $\B$ is  dense in $\calH$. Moreover, when additionally $G$ is abelian,  as a consequence of  \cite[Lemma 7.2]{FG2},  $\calA=\B$.

 Fixing an arbitrary non-zero element $g\in\calA$, the space $\M^1(G)$ is given by
$$\M^1(G):=\{ f\in\calH : \V_gf\in L^1(G)\},$$
and it is called a {\it modulation space}. 
It is a Banach space with the norm $\|f\|_{\M^1}:=\|\V_gf\|_{L^1}$. The set $\M^1$ is independent of the choice of $g$; i.e., different vectors in $\calA$ give the same space with equivalent norms (see, for instance, \cite[Theorem 4.2]{FG1}).   Considering the topological dual of $\M^1$, denoted as $(\M^1(G))'$,  and $p\in [1,\infty]$, it is said  that $f\in \M^p(G)$, the \emph{modulation space of order $p$}, if $f\in(\M^1(G))'$ and $\|\V_gf\|_{L^p(G)}<+\infty$. These spaces are  Banach spaces with the norm $\|f\|_{\M^p}=\|\V_gf\|_{L^p}$, and they are also  independent of the choice of the window $g$. For more details on these spaces we refer to \cite[Section 4]{FG1}.

It is known \cite[Theorem 8.1]{FG2} that for general locally compact groups, when $g\in\B$, $\V_g$ maps   the space $\M^p$ into $W(\calC,\ell^p)$. We will prove now that, when $G$  has a compact open subgroup, $\V_g$ maps   $\M^p$ into $W(\calC,\ell^p)$ continuously.

\begin{proposition} \label{prop} Let $G$ be a locally compact group with  a compact open subgroup $H$, $\pi:G\to\mathcal{U}(\calH)$ an irreducible unitary  representation of $G$,  $g\in \B$, and $1\leq p\leq +\infty$. If $f\in\M^p(G)$, then $\V_gf\in W(\calC,\ell^p)$. Furthermore, there exists a constant $K>0$ such that
\begin{equation}\label{desig}\|\V_gf\|_{W(\calC,\ell^p)}\leq K\|f\|_{\M^p} \end{equation}
for every $f\in\M^p(G)$.
\end{proposition}

\begin{proof} By  \cite[Theorem 8.1]{FG2} we know that for $f\in \M^p(G)$ we have $\V_gf\in W(\calC,L^p)$, and that  for each set $X=\{x_i\}_{i\in I}$ which is a section of $G/H$, the linear operator given by
$$R_X: f\mapsto (\V_gf(x_i))_{i\in I}$$
maps $\M^p(G)$ to $\ell^p(I)$ continuously; i.e., there exists $K_X>0$ such that
$$\|(\V_gf(x_i))_{i\in I}\|_{\ell^p}\leq K_X \|f\|_{\M^p} \hspace{.5cm} \forall f\in \M^p(G).$$

Consider $\F=\{R_X : \mbox{X is a section of $G/H$}\}$ and fix $f\in \M^p(G)$. Then, since $\V_gf$ is continuous, we have that 
$$\sup_{R_X\in\F}\|R_X(f)\|_{\ell^p}=\sup_{X=\{x_i\}_{i\in I}} \|(\V_gf(x_i))_{i\in I}\|_{\ell^p}=\|(\V_gf(\tilde{x_i}))_{i\in I}\|_{\ell^p}$$
where $X$ is a section of $G/H$ and, for each $i\in I$, $\tilde{x_i}\in H+x_i$ is a point that maximizes $|\V_gf|$ on $H+x_i$. Then, by the uniform boundedness principle there exists $K>0$ such that
$$\sup_{R_X\in\F} \|R_X\|=\sup_{R_X\in\F} K_X\leq K.$$
Consequently, we have for all sections $\{x_i\}_{i\in I}$ of $G/H$ that the following inequality holds
$$\|(\V_gf(x_i))_{i\in I}\|_{\ell^p}\leq K \|f\|_{\M^p},$$
for every $f\in  \M^p(G)$. Now, note that, for each $f \in \M^p(G)$, $\|\V_gf\|_{W(\calC,\ell^p)}=\|(\V_gf(\tilde{x_i}))_{i\in I}\|_{\ell^p}$ for a proper section $\tilde{X}=\{\tilde{x_i}\}_{i\in I}$.
Thus, 
$$\|\V_gf\|_{W(\calC,\ell^p)}\leq K \|f\|_{\M^p} \hspace{.5cm} \forall f\in \M^p(G).$$
\end{proof}

We shall  see now that there is a valid version of Proposition \ref{prop} for projective representations. 

For this, let $G$ be a locally compact group and recall that a {\it projective representation} is a continuous mapping  $\Pi:G\to\calU(\calH)$ for which there exists a continuous function $\alpha:G\times G\to \mathbb T$, called a {\it 2-cocycle}, such that $\Pi(x)\Pi(y)=\alpha(x,y)\Pi(xy)$. It is usual to call $\Pi$ an {\it $\alpha$-projective representation} to emphasize the dependence of $\Pi$ on $\alpha$. 
As we did for unitary representations, we define the {\it generalized wavelet transform}  corresponding to a projective representation as
$$\V^\Pi_gf(x):=\langle f, \Pi(x)g\rangle,$$
for $f, g\in\calH$.

Every projective representation of $G$  induces a unitary representation on the \emph{Mackey group} associated to $G$. The last is defined as follows: if $G$ is a locally compact group and $\alpha$ is a 2-cocycle, as a topological space, the Mackey group is $G\times\mathbb{T}$ with the product is given by
$$(x_1,\tau_1)(x_2,\tau_2)=(x_1x_2,\tau_1\tau_2\alpha(x_1,x_2)),$$
for $x_1,x_2\in G$ and $\tau_1,\tau_2\in\mathbb T$. The Mackey group associated to $G$ is a locally compact group and its Haar measure is given by the product of the Haar measures on $G$ and $\mathbb T$.
Then, for a $\alpha$-projective representation $\Pi:G\to\calU(\calH)$, define $\pi:G\times \mathbb T\to\calU(\calH)$ as 
\begin{equation}\label{def:unitaryMackey}
\pi(x,\tau):=\tau\Pi(x),
\end{equation}
for $\tau\in\mathbb T, x\in G$. This mapping $\pi$ turns out to be a unitary representation, and it is irreducible when $\Pi$ is.  
Note that for every $f,g\in\calH$, $x\in G$ and $\tau\in\mathbb T$ we have
$$\V^\Pi_gf(x)=\tau\langle f, \tau \Pi(x)g\rangle=\tau\langle f, \pi(x,\tau)g\rangle=\tau\V_gf(x,\tau),$$
with $\V_gf$ being the generalized wavelet transform associated to $\pi$. 
As a consequence, the sets $\calA$ and $\B$ given in \eqref{eq:DefA} and \eqref{eq:DefB} respectively, remain equal if we use the generalized wavelet transform induced by $\Pi$ instead of the one induced by the unitary representation given by \eqref{def:unitaryMackey}.
Then, the same holds for modulation spaces. 

Therefore, we obtain the corresponding version of Proposition \ref{prop} for projective representations. This result is a particular case of \cite[Theorem 2.5]{GS07}, with a significant distinction being our demonstration that the inclusion given by the wavelet transform is continuous.

\begin{theorem}\label{projective-inequality} Let $G$ be a locally compact  group with  a compact open subgroup, $\Pi:G\to\mathcal{U}(\calH)$ a irreducible  projective representation of $G$,  $g\in \B$ and $1\leq p\leq +\infty$. If $f\in\M^p(G)$  then $\V^\Pi_gf\in W(\calC,\ell^p)$. Furthermore, there exists a constant $K>0$ such that
\begin{equation}\label{desig}\|\V^\Pi_gf\|_{W(\calC,\ell^p)}\leq K\|f\|_{\M^p} \end{equation}
for every $f\in\M^p(G)$.
\end{theorem}

\begin{proof} 
Note that if $H\subseteq G$ is a compact open subgroup of $G$, then $H\times \mathbb T$ is a compact open subgroup of the Mackey group of $G$.  Also note that if $C$ is a section of $G/H$, then $C\times \{1\}$ is a section of $(G\times \mathbb T)/(H\times \mathbb T)$. Therefore, for each $c\in  C$ we have
$$\|\chi_{Hc}\cdot \V^\Pi_gf\|_\infty=\|\chi_{Hc\times \mathbb T}\cdot \V_gf\|_\infty,$$
and then
$$\|\V^\Pi_gf\|_{W(\calC,\ell^p)}=\|\V_gf\|_{W(\calC(G\times \mathbb T),\ell^p)}.$$
The result follows by applying  Proposition \ref{prop}.
\end{proof}

\section{Reproducing properties of Gabor systems}\label{sec:Gabor}

In this section we study Bessel and frame conditions on Gabor systems in terms of density.
To be precise, let us  fix $G$, an LCA group. For every $x\in G$,  the {\it translation operator} by $x$ of a function $f\in L^2(G)$ is given by 
$$T_xf(y)=f(y-x),\,\,\textrm{ for } m_G\textrm{-a.e. } y\in G.$$  

For $\xi\in\wh G$, the {\it modulation operator} by $\xi$ of a function $f\in L^2(G)$ is defined by 
$$M_\xi f(y)=\langle y,\xi\rangle f(y), \,\,\textrm{ for } m_G\textrm{-a.e. } y\in G..$$ 

Now, given a function $\varphi\in L^2(G)$ and a set $\Lambda\subseteq  G\times\widehat G$, we define the {\it  Gabor system generated by $\varphi$ and $\Lambda$} as 
$$S(\varphi,\Lambda)=\{M_\xi T_x\varphi\}_{(x,\xi)\in \Lambda}.$$

In order to establish frame conditions on $S(\varphi,\Lambda)$ in terms of density of $\Lambda$, we assume that $G$ has a compact open subgroup $H$, and we fix $A\in \Aut(G)$ expansive with respect to $H$ such that $A^*$ is expansive with respect to $H^\perp$.  Then, the densities of $\Lambda\subseteq G\times \wh G$ are 
$$D^+(\Lambda):=\limsup_{n\to +\infty}\frac1{|A|^{2n}}\max_{(x, \xi)\in G\times \wh G}\#\{\Lambda\cap Q_n(x, \xi)\},$$ and 
$$D^-(\Lambda):=\liminf_{n\to +\infty}\frac1{|A|^{2n}}\min_{(x, \xi)\in G\times \wh G}\#\{\Lambda\cap Q_n(x, \xi)\},$$
where $Q_n(x,\xi)$ is defined as in \eqref{def:bola}.

On the other hand,  note that translation and modulation operators are unitary in $L^2(G)$, and they 
satisfy the intertwining relationship $M_\xi T_x f=\langle x,\xi\rangle T_x M_\xi f$ for all $x\in G$, $\xi\in\widehat G$ and for all $f\in L^2(G)$. Thus, the Gabor representation $\Pi:G\times\wh{G}\to\mathcal{U}(L^2(G))$ given by 
\begin{equation}\label{eq:Gabor-representation}
\Pi(x,\xi):=M_{\xi}T_x
\end{equation} is an irreducible projective representation with 2-cocycle given by $\alpha((x_1, \xi_1), (x_2, \xi_2))= \langle x_1,\xi_2\rangle$. Then, we can make use of the tools described in the previous section. In particular, we have defined the well-know {\it short-time Fourier transform}
$$V_g f(x,\xi):=\langle f, M_\xi T_x g \rangle=\V^\Pi_gf(x,\xi),$$
where, $f,g\in L^2(G)$, and $(x,\xi)\in G\times\wh{G}$. For  fixed $f,g\in L^2(G)$, $V_gf$ is well-defined and continuous on $G\times\widehat G$. 

When we consider translations along a section of $G/H$ and modulations along a section of $\wh{G}/H^{\perp}$ of the function $\chi_H$, it turns out that the obtained Gabor system  is an ortonormal basis for $L^2(G)$. See  \cite[Theorem 2.7, Case II]{GS07} for a proof of this fact. 
 
\begin{lemma} \label{bon} Let $C$ and $D$ be sets of coset representatives of $G/H$ and $\wh{G}/H^{\perp}$ respectively, and consider $\Lambda=C\times D$. 
Then, $S(\chi_H, \Lambda)$ is an orthonormal basis for $L^2(G)$.
\end{lemma}

\subsection{Bessel sequences} In this section we show one necessary and one sufficient condition for a Gabor system to be a Bessel sequence of $L^2(G)$.

We begin by proving that if a Gabor system is a Bessel sequence, then $\Lambda$ must have finite upper density. This was  proved before for $\R^d$   in \cite[Theorem 3.1]{CDH99}. 

\begin{theorem}\label{thm:density-Bessel-one-generator}
Let $\varphi\in L^2(G)$ and let $\Lambda\subseteq  G\times\widehat G$. If $S(\varphi,\Lambda)$ is a Bessel sequence in $L^2(G)$, then $D^+(\Lambda)<\infty$.
\end{theorem}

\begin{proof}
Let $f\in L^2(G)$ with  $\|f\|_2=1$ such that $\langle \varphi, f \rangle\neq 0$ and define $A_\varphi f:G\times\widehat G\to\mathbb R_{\geq 0}$ $$A_\varphi f(x,\xi):=|V_\varphi  f(x,\xi)|.$$
Then, $A_\varphi f$ is continuous on $G\times\widehat G$. 

As   $A_\varphi f\neq 0$, there exists $(x_0,\xi_0)\in G\times\widehat G$ and $n_0\in\Z$ such that $\eta :=\inf\{A_\varphi f(x,\xi):\,(x,\xi)\in Q_{n_0}(x_0,\xi_0)\}>0$.

If we had $D^+(\Lambda)=\infty$, by Lemma \ref{lem:density} for each $N>0$ there should exist 
some $(x_N,\xi_N)$ such that $\#\{\Lambda\cap Q_{n_0}(x_N,\xi_N)\}\geq N$ .

Now, note that if $(x,\xi)\in Q_{n_0}(x_N,\xi_N)$, then 
$(x,\xi)-(x_N,\xi_N)+(x_0,\xi_0)\in Q_{n_0}(x_0,\xi_0)$ and thus
$$ |\langle f, M_{\xi-\xi_N+\xi_0} T_{x-x_N+x_0} \varphi\rangle |\geq\eta.$$
Since
$ |\langle f, M_{\xi-\xi_N+\xi_0} T_{x-x_N+x_0} \varphi\rangle|=|\langle M_{\xi_0-\xi_N} T_{x_0-x_N} f,  M_\xi T_x\varphi  \rangle |$
we have that  
$$
\sum_{(x,\xi)\in \Lambda\cap Q_{n_0}(x_N,\xi_N)}|\langle M_{\xi_0-\xi_N} T_{x_0-x_N} f,  M_\xi T_x\varphi  \rangle |^2
\geq\eta^2N,
$$ for all $N\in\N$. Hence, $S(\varphi,\Lambda)$ can not be a Bessel sequence because 
$\|M_{\xi_0-\xi_N} T_{x_0-x_N} f\|_2=\|f\|_2=1$.
\end{proof}

The above theorem extends to a finite union of Gabor systems. More precisely, let $\Lambda_1,\dots, \Lambda_r\subseteq G\times \widehat G$ be sequences each indexed in $I_1,\dots, I_r$, respectively. That is, $\Lambda_k=\{(x_{i,k},\xi_{i,k})\}_{i\in I_k}$, for $1\leq k\leq r$.  Define $I=\{(i,k):\,i\in I_k, 1\leq k\leq r\}$ and $\Lambda$
as the sequence $\{ (x_{i,k},\xi_{i,k}): (i,k)\in I\}$. By abuse of notation, we will simply write $\Lambda=\bigcup_{k=1}^r \Lambda_k$ and say that $\Lambda$ is the {\it disjoint union of $\Lambda_1,\dots, \Lambda_r$}. Once this is clear, we can state the desired result. 

\begin{theorem}\label{thm:density-Bessel-multi-generators}
For $1\leq k\leq r$, let $\varphi_k\in L^2(G)$ and $\Lambda_k\subseteq G\times \widehat G$ a sequence. Consider $\Lambda$ the disjoint union of $\Lambda_1,\dots, \Lambda_r$. 
If $\bigcup_{k=1}^rS(\varphi_k, \Lambda_k)$ is a Bessel sequence, then $D^+(\Lambda)<\infty$. 
\end{theorem}

\begin{proof}
Note that since $\bigcup_{k=1}^rS(\varphi_k, \Lambda_k)$ is a Bessel sequence, so is $S(\varphi_k, \Lambda_k)$ for every $1\leq k\leq r$. Then, by Theorem \ref{thm:density-Bessel-one-generator} we have that  $D^+(\Lambda_k)<\infty$ for every $1\leq k\leq r$.

Now, since $\Lambda$ is the disjoint union of  $\Lambda_1,\dots, \Lambda_r$, for each $n\in\Z$ and $(x,\xi)\in G\times \widehat G$,
$$\#\{\Lambda\cap Q_n(x,\xi)\}=\sum_{k=1}^r \#\{\Lambda_k\cap Q_n(x,\xi)\},$$
and as a consequence
$$\sum_{k=1}^r D^-(\Lambda_k)\leq D^-(\Lambda)\leq D^+(\Lambda)\leq \sum_{k=1}^r D^+(\Lambda_k).$$
From here the conclusion follows. 
\end{proof}

In what follows we shall prove a weaker converse of Theorem  \ref{thm:density-Bessel-one-generator}. For this to be true, we have to assume that the generating function of the Gabor system must have a particular decay; that is, it must be a function of $ \M^1(G)$.

As a consequence of  \cite[Theorem 4.7]{JAK} (see also \cite[Theorem 11]{fe81-2}) we have that, when the representation involved 
is the Gabor representation \eqref{eq:Gabor-representation}, $\calA=\B=\M^1(G)$, and then,  Theorem \ref{projective-inequality} holds for $g\in \M^1(G)$.

Then, we have the  following result, which is a generalization of \cite[Theorem 12]{H07} to our setting.

\begin{theorem}\label{thm:Bessel} Let $\varphi\in \M^1(G)$, $\varphi\neq0$ and $\Lambda\subseteq G\times\wh{G}$ any sequence with $D^+(\Lambda)<+\infty$. Then $S(\varphi,\Lambda)$ is a Bessel sequence.
\end{theorem}
\begin{proof} By Lemma \ref{lem:density}, since $D^+(\Lambda)<+\infty$ there exists $N_0>0$ such that $\#\{ \Lambda\cap Q_0(c,d)\}\leq N_0$ for all $(c,d)\in C\times D$. Then for $f\in L^2(G)$,
\begin{align*}
\sum_{(x,\xi)\in\La} |\langle f, M_\xi T_x \p \rangle|^2 &=\sum_{(x,\xi)\in\La} |V_\p f(x,\xi)|^2 \\
&=\sum_{(c,d)\in C\times D} \sum_{(x,\xi)\in\La\cap Q_0(c,d)}|V_\p f(x,\xi)|^2 \\
&\leq \sum_{(c,d)\in C\times D} N_0 \sup_{(x,\xi)\in Q_0(c,d)} |V_\p f(x,\xi)|^2 \\
&= N_0 \|V_\p f\|_{W(\calC,\ell^2)}^2\leq N_0  K \|f\|_{\M^2}^2,
\end{align*}
where the last inequality is the result of Theorem \ref{projective-inequality}. Finally, 
since $\|f\|_{\M^2}=\|V_\varphi f\|_2$ and by the well-known orthogonality relationship of the short-time Fourier transform (see \cite[Section 2.2]{FG1}, $\|V_\varphi f\|_2 =\|f\|_2\|\varphi\|_2$, we conclude the result. 
\end{proof}

It is known that every locally compact abelian group $G$  is algebraically and topologically isomorphic to $\R^d\times G_0$, where $G_0$ is an LCA group with an open compact subgroup (see, e.g.,  \cite{DE14}). For the case where $G_0$ is  an expansible group
we can combine Theorem \ref{thm:Bessel} with  \cite[Theorem 12]{H07} to prove a similar statement for the product group $\R^d\times G_0$.

\begin{proposition} Let $G_0$ be an expansible LCA group.  Take  $g_1\in M^1(\R^d)$, $g_2\in M^1(G_0)$ and $\Lambda=\Lambda_1\times\Lambda_2$, where $\Lambda_1\subseteq\R^{2d}$ and $\Lambda_2\subseteq G_0\times\wh{G_0}$. 
Suppose that $\Lambda_1$ and $\Lambda_2$ have finite upper density, where the density of 
$\Lambda_1$ is the usual Beurling density defined in the introduction, and the density of $\Lambda_2$ is as in Definition \ref{def:density}.
If we consider $g\in L^2(\mathbb R^d\times G_0)$ given by  $g=g_1\otimes g_2$,  then $S(g,\Lambda)$ is a Bessel sequence for $L^2(\R^d\times G_0)$.
\end{proposition}

\begin{proof} First observe that for $(x_1,x_2,\xi_1,\xi_2)\in (\R^d\times G_0)\times (\R^d\times\wh{G_0})$ and $f\in L^2(\R^d\times G_0)$, we have 
\begin{align*}
V_gf(x_1,x_2,\xi_1,\xi_2)&=\int_{\R^d}\int_{G_0} f(y_1,y_2) \overline{g(y_1-x_1,y_2-x_2)}\overline{\langle y_1, \xi_1\rangle}\overline{\langle y_2, \xi_2\rangle}dm_{G_0}(y_2)dy_1 \\
&=\int_{\R^d}\int_{G_0} f(y_1,y_2) \overline{g_1(y_1-x_1)g_2(y_2-x_2)}\overline{\langle y_1, \xi_1\rangle}\overline{\langle y_2, \xi_2\rangle}dm_{G_0}(y_2)dy_1 \\
&\int_{\R^d}\overline{g_1(y_1-x_1)\langle y_1, \xi_1\rangle}\left(\int_{G_0}f(y_1, y_2)\overline{g_2(y_2)\langle y_2, \xi_2\rangle}dm_{G_0}(y_2)\right)dy_1 \\
&=V_{g_1}\left(V_{g_2}f_{\bullet}(x_2,\xi_2)\right)(x_1,\xi_1),
\end{align*}

where by $f_{\bullet}$ we denote the function the is obtained when the first variable of $f$ is fixed. 

\noindent On the other hand, since both upper densities are finite,  we know from Lemma \ref{lem:density} and \cite[Lemma 2.3]{CDH99} that there exist $N_0$ and $N_1 \in \N$ such that
\begin{align*}
&\#(\Lambda_2\cap Q_0(x,\xi))\leq N_0 \,\mbox{for all}\,\, (x,\xi)\in G_0\times\wh{G_0}\,\,\mbox{and}\\
&\#(\Lambda_1\cap B_1(t))\leq N_1 \,\mbox{for all}\,\, t\in\R^{2d}, 
\end{align*}
where $B_1(0)=[0,1]^{2d}$ and $B_1(t)=B_1(0)+t$.

\noindent Putting this together we obtain
\begin{align*}
\sum_{(x_1,x_2,\xi_1,\xi_2)\in\Lambda}&|V_gf(x_1,x_2,\xi_1,\xi_2)|^2 = \sum_{(x_2,\xi_2)\in\Lambda_2}\sum_{(x_1,\xi_1)\in\Lambda_1}|V_{g_1}\left(V_{g_2}f_{\bullet}(x_2,\xi_2)\right)(x_1,\xi_1)|^2\\
&= \sum_{(x_2,\xi_2)\in \Lambda_2} \sum_{j\in\Z^{2d}}\sum_{(x_1,\xi_1)\in \Lambda_1\cap B_1(j)}|V_{g_1}\left(V_{g_2}f_{\bullet}(x_2,\xi_2)\right)(x_1,\xi_1)|^2 \\
&\leq \sum_{(x_2,\xi_2)\in \Lambda_2} \sum_{j\in\Z^{2d}} N_1 \sup_{(x_1,\xi_1)\in \Lambda_1\cap B_1(j)}|V_{g_1}\left(V_{g_2}f_{\bullet}(x_2,\xi_2)\right)(x_1,\xi_1)|^2 \\
&= N_1 \sum_{(x_2,\xi_2)\in \Lambda_2} \|V_{g_1}(V_{g_2}f_{\bullet}(x_2,\xi_2))\|_{W(\R^d)}^2 \\
&\leq N_1 \sum_{(x_2,\xi_2)\in\Lambda_2} \|V_{g_2}f_{\bullet}(x_2,\xi_2)\|_{L^2(\R^d)}^2 K_1 \|g_1\|_{L^2(\R^d)}^2, 
\end{align*}
where in the last inequality we used \cite[Theorem 12]{H07} and the fact that $V_{g_2}f_{\bullet}(x_2,\xi_2)\in L^2(\R^d)$ for a.e.\ $(x_2,\xi_2)\in G_0\times\wh{G_0}$.
Now, using Fubini and Theorem \ref{projective-inequality} we have 
\begin{align*}
N_1K_1 \|g_1\|_{L^2(\R^d)}^2 \sum_{(x_2,\xi_2)\in\Lambda_2} &\|V_{g_2}f_{\bullet}(x_2,\xi_2)\|_{L^2(\R^d)}^2 = K_0 \sum_{(x_2,\xi_2)\in\Lambda_2} \int_{\R^d}|V_{g_2}f_{y_1}(x_2,\xi_2)|^2dy_1
\\&= K_0 \int_{\R^d} \sum_{(c,d)\in C\times D}\sum_{(x_2,\xi_2)\in\Lambda_2\cap Q_0(c,d)}|V_{g_2}f_{y_1}(x_2,\xi_2)|^2dy_1\\
&\leq K_0 \int_{\R^d} \sum_{(c,d)\in C\times D} N_0 \sup_{(x_2,\xi_2)\in\Lambda_2\cap Q_0(c,d)}|V_{g_2}f_{y_1}(x_2,\xi_2)|^2 dy_1\\
&= K_0 N_0 \int_{\R^d} \|V_{g_2}f_{y_1}\|_{W(G_0)}^2 dy_1 \\
&\leq K \int_{\R^d} \|f_{y_1}\|_{L^2(G_0)}^2 dy_1 \\
&= K \int_{\R^d}\int_{G_0} |f(y_1,y_2)|^2 dy_2dy_1 = K \|f\|^2_{L^2(\R^d\times G_0)}.
\end{align*}
Since $\sum_{(x_1,x_2,\xi_1,\xi_2)\in\La} |\langle f, M_{(\xi_1,\xi_2)} T_{(x_1,x_2)} g \rangle|^2 =\sum_{(x_1,x_2,\xi_1,\xi_2)\in\La} |V_gf(x_1,x_2,\xi_1,\xi_2)|^2 $, the above computation shows that $S(g,\Lambda)$ is a Bessel sequence. 
\end{proof}

 As in the Euclidean case (see for instance \cite[Theorem 1.1]{CDH99} and \cite[Theorem 5.7]{H08}), it is possible to prove a Density Theorem for Gabor frames in this context as well. This is a consequence of  \cite[Corollary 3.4]{CV23} which has a very sophisticated proof. However, in our setting, we can easily deduce it from \cite[Theorem 4]{GR08}. We include the reasoning for the sake of completeness.

 \begin{theorem}\label{thm: comparison} Let $\Lambda,\Delta\subseteq G\times\wh{G}$ be uniformly separated sets and $\varphi, h\in L^2(G)$ such that $S(\varphi,\Lambda)$ and $S(h,\Delta)$ are a frame and a Riesz basis for $L^2(G)$ respectively. Then 
 $$D^-(\Delta)\leq D^-(\Lambda) \,\,\,\mbox{and}\,\,\,D^+(\Delta)\leq D^+(\Lambda).$$ 
 \end{theorem}

 \begin{proof}
First, we will modify $h\in L^2(G)$ to be the scalar multiple $\tilde{h}$ such that $\|\tilde{h}\|^2_2=B$, where $B$ is the upper frame bound of $S(\varphi,\Lambda)$. With this replacement $S(\tilde{h},\Delta)$ is still a Riesz basis for $L^2(G)$. By \cite[Theorem 4]{GR08} for a given $\varepsilon>0$ there exists a compact set $L\subseteq G\times\wh{G}$ such that
\begin{equation} \label{eq:comparison}
    (1-\varepsilon)\#\{\Delta\cap (K+(x,\xi))\}\leq \#\{\Lambda\cap (K+L+(x,\xi))\},
\end{equation}
for every compact set $K\subseteq G\times\wh{G}$ and for every $(x,\xi)\in G\times\wh{G}$. 

\noindent Since both $G$ and $\wh G$ are expansible and $L$ is compact, there exists $n_0\in\N$ such that $Q_n(0,0)+L=Q_n(0,0)$ for all $n\geq n_0$. Then,  \eqref{eq:comparison} implies that 
\begin{equation*}
    (1-\varepsilon)\#\{\Delta\cap Q_n(x,\xi)\}\leq \#\{\Lambda\cap Q_n(x,\xi)\}
\end{equation*}
 for all $(x,\xi)\in G\times\wh{G}$ and $n\geq n_0.$ Dividing by $|A|^{2n}$ and taking $\limsup$ (or respectively $\liminf$) we obtain
 $$(1-\varepsilon)D^{\pm}(\Delta)\leq D^{\pm}(\Lambda)$$
 for every $\varepsilon>0$. From here the result follows.
 \end{proof}

 \noindent  With the above comparison result, it is possible to deduce a density result. 
 
 \begin{corollary}
     Let $\Lambda\subset G\times\wh{G}$ be a uniformly separated set and $\varphi\in L^2(G)$. 
     \begin{enumerate}
         \item[(i)] If $S(\varphi,\Lambda)$ is a frame for $L^2(G)$ then $D^-(\Lambda)\geq 1$;
         \item[(ii)] If $S(\varphi,\Lambda)$ is a Riesz basis for $L^2(G)$ then $D(\Lambda)=1.$
     \end{enumerate}
 \end{corollary}

 \begin{proof}
By \Cref{bon}, the system $S(\chi_H,\Delta)$ with $\Delta$ being a section of $G/H\times\wh{G}/H^{\perp}$ is an orthonormal  basis for $L^2(G)$ and, as we calculate in \Cref{ex:densidadC}, $D(\Delta)=1$. Let us know consider $\Lambda$ and $\varphi$ as in the statement. If $S(\varphi,\Lambda)$ is a frame for $L^2(G)$ then by \Cref{thm: comparison} we have $D^-(\Lambda)\geq D^-(\Delta)=1$. If, on the other hand, $S(\varphi,\Lambda)$ is a Riesz basis for $L^2(G)$ then we apply \Cref{thm: comparison} twice interchanging the role of each system to obtain $D(\Lambda)=1$. \end{proof}

We finish by showing that in  our setting, we can construct Gabor frames generated by sets with density $|A|^{2k}$ for every $k\geq0$.

\begin{example}\label{Ex:bon}  Let $G$ be an expansible group and $A$ an expansive automorphism with respect to $H$, an open compact subgroup of $G$. Let $k\geq0$ and consider $\varphi:=|A|^{k/2}\chi_{A^{-k}H}$ and $\Lambda =\Delta\times\Gamma$ with $\Delta$ a section of $G/A^{-k}H$ and  $\Gamma$ a section of $\wh G/A^{*-k}H^{\perp}$. Then we can write $\Lambda=\Lambda_1+\Lambda_2$ where $\Lambda_1$ is a section of $G/H\times\wh{G}/H^{\perp}$ and $\Lambda_2$ is a section of $H/A^{-k}H\times H^{\perp}/A^{*-k}H^{\perp}$.

We want to calculate $D(\Lambda)$. In that sense, for $n\in\N$ we have
\begin{align*} 
\max_{(x,\xi)\in G\times\wh{G}} \#(Q_n(x,\xi)\cap \Lambda)&\leq \sum_{\lambda_2\in \Lambda_2} \max_{(x,\xi)\in G\times\wh{G}} \#(Q_{n}(x,\xi)\cap(\Lambda_1+\lambda_2)) \\
&= \sum_{\lambda_2\in \Lambda_2} \max_{(x,\xi)\in G\times\wh{G}} \#(Q_{n}(x,\xi)\cap \Lambda_1) \\ &= \#\Lambda_2 \max_{(x,\xi)\in G\times\wh{G}} \#(Q_{n}(x,\xi)\cap \Lambda_1) \\
&= |A|^{2k} \max_{(x,\xi)\in G\times \wh{G}}\# (Q_{n}(x,\xi)\cap \Lambda_1).
\end{align*}
Dividing by $|A|^{2n}$, taking $\limsup$ and applying Lemma \ref{ex:densidadC} we obtain 
\begin{align*}
D^+(\Lambda)\leq |A|^{2k} D^+(\Lambda_1)= |A|^{2k}. 
\end{align*}
By doing an analogous calculation with the minimum, we obtain $D^-(\Lambda)\geq |A|^{2k} D^-(\Lambda_1)= |A|^{2k} $ and therefore $D(\Lambda)=|A|^{2k}$.

\noindent On the other hand, we can also write  $\Gamma=\Gamma_1+\Gamma_2$ where $\Gamma_1$ and $\Gamma_2$ are sections of $\wh G/A^{*k}H^{\perp}$ and $A^{*k}H^{\perp}/A^{*-k}H^{\perp}$ respectively. As a consequence, by Lemma \ref{bon},  for every $\gamma_2\in \Gamma_2$, the system $S(\varphi,\Delta\times(\Gamma_1+\gamma_2)) $ is an orthonormal basis for $L^2(G)$. Then, $S(\varphi,\Lambda)$ is a tight frame with bound $\#\Gamma_2=|A|^{2k}$.  

\end{example}

\begin{lemma}\label{lemma: density of sections}
    Let $G$, $A$ and $H$ as in the prior example. Consider $\varphi=\chi_H$, and for $k\in\Z$ let $\Lambda_k$  be a section of $G/H\times \wh{G}/A^{*-k}H^{\perp}$. Then the following holds
    \begin{enumerate}
        \item[(i)] if $k=0$ then $S(\varphi, \Lambda_0)$ is an orthonormal basis for $L^2(G)$,
        \item[(ii)] if $k>0$ then $S(\varphi, \Lambda_k)$ is a tight frame for $L^2(G)$ with constant $|A|^k$,
        \item[(iii)] and if $k<0$ then $S(\varphi, \Lambda_k)$ is an incomplete system for $L^2(G)$. 
    \end{enumerate}
\end{lemma}

\begin{proof}
When $k=0$ the set $\Lambda_0$ is a section of $G/H\times\wh{G}/H^{\perp}$ and by Lemma \ref{bon} is an orthonormal basis. If $k>0$ the same procedure as in the prior example can be used to prove that $S(\varphi,\Lambda_k)$ is a tight frame with constant $|A|^k$. For the remaining case, let us remember  $H^{\perp}\subseteq A^{*-k}H^{\perp}$ and choose  $\Gamma$ to be a section of $G/H\times\wh{G}/H^{\perp}$ containing $\Lambda_k$ properly. Since by Lemma \ref{bon} the system  $S(\varphi, \Gamma)$ is an orthonormal basis for $L^2(G)$ and  $S(\varphi, \Lambda)\subsetneq S(\varphi, \Gamma)$, the statement is proven.
\end{proof}

\nocite{*}

\section{Acknowledgements}
	This research was supported by grants: V.P. and R.N. were supported by PICT 2018-3399 (ANPCyT), PICT 2019-03968 (ANPCyT) and CONICET PIP 11220210100087.


\end{document}